\documentclass[english, 12pt]{amsart}
\usepackage{amssymb}
\usepackage{amsfonts}
\usepackage{amscd}
\usepackage{color}
\usepackage{tikz}

\providecommand{\noopsort}[1]{}

\usepackage{hyperref}
\usepackage{cleveref}

\newbox\removebox
\newcommand\remove[2][blue]{%
\setbox\removebox=\ifmmode\hbox{$#2$}\else\hbox{#2}\fi%
\leavevmode
\rlap{\textcolor{#1}{\vrule height0.8ex depth-0.5ex width\wd\removebox}}%
\box\removebox
}
\long\def\bigremove#1{%
\par\setbox\removebox=\vbox{#1}%
\vbox{%
\vbox to0pt{\hbox{\tikz\draw[color=blue,thick] (0,0) -- (\wd\removebox,-\ht\removebox)  (\wd\removebox,0) -- (0,-\ht\removebox);}}
\box\removebox
}
}

\usepackage{mathrsfs} 

\newcommand{\cCe}{\cC^{\mathrm{e}}}
\newcommand{\cCsq}{\cC^{\square}}

\newcommand{\neww}{_{\mathrm{new}}}
\newcommand{\maxx}{_{\mathrm{max}}}

\newcommand{\statement}[1]{\text{\textup{#1}}}
\newcommand{\EZ}{\statement{(Eval=0)}}
\newcommand{\FZ}{\statement{(Fct=0)}}
\newcommand{\CZ}{\statement{(Coeff=0)}}
\newcommand{\ET}{\statement{(EvalTor)}}
\newcommand{\FT}{\statement{(FctTor)}}
\newcommand{\CT}{\statement{(CoeffTor)}}

\def\VG{\mathrm{VG}}
\newcommand{\RF}{{\rm RF}}

\def\gTPas
{{\rm gDP}}

\def\TPres
{{\rm PRES}}

\def\11{{\mathbf 1}}

\def\LL{{\mathbb L}}

\def\NN{{\mathbb N}}

\def\ZZ{{\mathbb Z}}

\def\cC{{\mathscr C}}

\def\cL{{\mathcal L}}

\def\cS{{\mathcal S}}

\newtheorem{thm}{Theorem}[section]
\newtheorem{lem}[thm]{Lemma}

\newtheorem{prop}[thm]{Proposition}

\theoremstyle{definition}

\newtheorem{notn}[thm]{Notation}

\newtheorem{example}[thm]{Example}

\newtheorem{def-prop}[thm]{Proposition-Definition}
\newtheorem{def-thm}[thm]{Theorem-Definition}
\newtheorem{def-lem}[thm]{Lemma-Definition}
\newtheorem{assumption}[thm]
{Assumption}

\theoremstyle{remark}
\newtheorem{remark}[subsection]
{Remark}
{Remark}

\Crefname{prop}{Proposition}{Propositions}
\Crefname{thm}{Theorem}{Theorems}
\Crefname{lem}{Lemma}{Lemmas}
\Crefname{cor}{Corollary}{Corollaries}
\Crefname{conj}{Conjecture}{Conjectures}
\Crefname{defn}{Definition}{Definitions}
\Crefname{defnsubsub}{Definition}{Definitions}

\theoremstyle{plain}

  {\par\medskip\noindent #1\par\begingroup%
    \advance\leftskip by 1em\advance\rightskip by 1em}%
  {\par\endgroup}

\usepackage[T1]{fontenc}
\usepackage{mathrsfs}

\definecolor{immi's color}{rgb}{.9,.3,0}

\begin{document}

\setcounter{tocdepth}{1} 

\author
{Raf Cluckers}

\address{Univ.~Lille,  CNRS, UMR 8524 - Laboratoire Paul Painlev\'e, F-59000 Lille, France, and
KU Leuven, Department of Mathematics, B-3001 Leu\-ven, Bel\-gium}
\email{Raf.Cluckers@univ-lille.fr}
\urladdr{http://rcluckers.perso.math.cnrs.fr/}

\author
{Immanuel Halupczok}
\address{Mathematisches Institut,
Universit\"atsstr. 1, 40225 D\"usseldorf,
Germany}
\email{math@karimmi.de}
\urladdr{http://www.immi.karimmi.de/en/}


\thanks{The authors deeply thank Floris Vermeulen for noting the mistake in Proposition 4.1 of \cite{CH-eval}. 
R.C. was partially supported by KU Leuven IF C16/23/010 and acknowledges the support of the CDP C2EMPI, as well as the French State under the France-2030 programme, the University of Lille, the Initiative of Excellence of the University of Lille, the European Metropolis of Lille for their funding and support of the R-CDP-24-004-C2EMPI project. The second author is partially supported by the research training group
\emph{GRK 2240: Algebro-Geometric Methods in Algebra, Arithmetic and Topology}, and by the individual research grant \emph{Archimedische und nicht-archimedische Stratifizierungen h\"oherer Ordnung}, both funded by the DFG}


\title[Corrigendum to `Evaluation of motivic functions, ...']{Corrigendum to `Evaluation of motivic functions, non-nullity, and integrability in fibers', Advances in Mathematics, Vol. 409, Part A, Paper No. 108635, 29 pages, doi:10.1016/j.aim.2022.108635 (2022)}

\begin{abstract}
We correct the statements and proofs of the (auxiliary) Propositions~4.1 and 4.2 of our paper `Evaluation of motivic functions, non-nullity, and integrability in fibers' in this journal (2022), and we explain how the proofs of the main results can be adapted to work with those corrected propositions.
\end{abstract}

\maketitle

The authors regret that the proof of the furthermore part of \cite[Proposition~4.1]{CH-eval} does not work in the presence of torsion. This then also affects one part of \cite[Proposition~4.2]{CH-eval}, where that furthermore part is used. In this corrigendum, we formulate and prove corrected versions of the wrong parts of these two propositions (as Propositions~\ref{p.4.1} and \ref{p.4.2}), and we explain how to adapt the proofs of the main results of
\cite{CH-eval} to use the corrected statements (as Propositions~\ref{p.T1} and \ref{p.T2}).
In Sections~\ref{s.corr} and \ref{s.proof}, we deal with the results related to nullity of motivic functions, namely \cite[Proposition~4.1]{CH-eval} and its application in the proof of \cite[Theorem~1]{CH-eval}. Section~\ref{s.int} deals with the statements related to ingerability, namely \cite[Proposition~4.2]{CH-eval} and its application in the proof of \cite[Theorem~2]{CH-eval}.

\section{Corrected statements concerning nullity}
\label{s.corr}

We use notation and terminology from \cite{CH-eval}. In particular, we work in valued fields using an expansion $\cL$ of the Denef-Pas language by constant symbols. As in \cite{CH-eval}, we need to work in a non-elementary class $\cS$ of valued fields, all of which are of the form $(k((t)), k, \ZZ)$, for some fields $k$ of characteristic $0$, and we emphasize this by writing ``$\cS$-definable''.
Recall that non-elementarity of the class $\cS$ is relevant for the notion of a ``point'' of an $\cS$-definable set $X$: A point is a pair $(K, x)$ with $K \in \cS$ and $x \in X_K$. For details about this and the related notation, we refer the reader to \cite[Sections 2, 3.1]{CH-eval}. However, for the purpose of this corrigendum, one could mostly pretend to just work in one fixed valued field of the form $k((t))$ all the time.

In \cite[Sections 3.2, 3.3]{CH-eval}, several classes of motivic functions are introduced, in line with the classes of motivic functions introduced in \cite{CLoes} and \cite{CLexp}. This entire corrigendum applies to the two classes $\cC$ and $\cCe$. We introduce a short hand notation to prove things for both of them simultaneously.

\begin{notn}
Given any $\cS$-definable set $X$, $\cCsq(X)$ either stands for $\cC(X)$ or for $\cCe(X)$. By a ``motivic function on $X$'', we mean an element of $\cCsq(X)$.
\end{notn}

Almost the entire corrigendum is about motivic functions of a specific shape. We fix this once and for all.

\begin{assumption}\label{a.f}
Let $A$ be any $\cS$-definable set and let
\begin{equation}\label{sum}
f = \sum_{(a,b) \in L}  c_{a,b}\cdot  g^a \cdot  \LL^{b\cdot  g}
\in \cCsq(\VG_{\ge 0}^r \times A)
\end{equation}
be a motivic function, where
$L \subset \NN^r \times \ZZ^r$ is a finite set, $c_{a,b} \in \cCsq(A)$, $g = (g_1, \dots, g_r)$ is the coordinate projection from $\VG_{\ge 0}^r \times A$ to $\VG_{\ge 0}^r$,  $g^a$ stands for $g_1^{a_1}\cdots g_{r}^{a_{r}}$ and
$b \cdot g$ stands for $b_1\cdot g_1 + \dots+ b_{r} \cdot g_{r}$.
We sometimes implicitly set $c_{a,b} = 0$ for $(a,b) \notin L$.
\end{assumption}

Consider the following statements about such $f$:
\begin{enumerate}
 \item[\EZ]
 For every point $x$ on $\VG_{\ge 0}^r$, the restriction of $f_{|\{x\} \times A}$
 is zero (as an element of $\cCsq(\{x\} \times A)$).
 \item[\FZ] $f = 0$ (as an element of $\cCsq(\VG_{\ge 0}^r \times A)$).
 \item[\CZ] For all $(a,b) \in L$, we have $c_{a,b} = 0$ (in $\cCsq(A)$).
\end{enumerate}

Clearly, we have $\CZ \Rightarrow \FZ \Rightarrow \EZ$.
The ``furthermore'' part of \cite[Proposition~4.1]{CH-eval} claims that one also has the reverse implication $\FZ\Rightarrow\CZ$. This is potentially wrong, as Example~\ref{ex:non-torsion} shows. (In the proof, we were not careful enough concerning torsion elements in $\cCsq(\VG_{\ge 0}^r \times A)$.) That furthermore part is used in two places, namely
the proofs of \cite[Proposition~4.2]{CH-eval} and \cite[Theorem~1]{CH-eval}.
For the moment, let us consider the latter one.
(Proposition~4.2 is related to integrability; we will come back to it in Section~\ref{s.int}.)

Theorem~1 of \cite{CH-eval} states, among others, that if a motivic function $f \in \cCsq(X)$ evaluates to $0$ at every point on $X$, then $f$ itself is $0$. Case~2 on \cite[p.~23]{CH-eval} proves Theorem~1 under the assumption that $X \subset \VG^R \times \RF^n$. This is reduced to the special case where $X = \VG_{\ge0}^r \times A$, for some $A \subset \RF^{N}$, and where $f$ is as
in Assumption~\ref{a.f}.
Using Case~1 on \cite[p.~22]{CH-eval}, one deduces (from the assumption that $f$ evaluates to $0$ at every point on $X$) that $f$ satisfies $\EZ$, so to finish the proof of Case~2, it remains to prove:

\begin{prop}\label{p.T1}
For $A, f, L,\dots$ as in Assumption~\ref{a.f}, we have $\EZ\Rightarrow\FZ$.
\end{prop}

This is where the furthermore part of \cite[Proposition~4.1]{CH-eval} is used.
Inspecting that proof carefully shows not only that it can be adapted to use the modified furthermore part that we state below as Proposition~\ref{p.4.1}, but moreover, that the key ingredients of the proof of Proposition~\ref{p.4.1} itself are also already contained there. We will nevertheless give detailed proofs of Propositions~\ref{p.T1} and \ref{p.4.1} in this corrigendum.

To state the corrected furthermore part, we introduce weaker version of the above statements about $f$, namely, where ``$=0$'' has been replaced by ``is torsion'':
\begin{enumerate}
 \item[\ET]
  For every point $x$ on $\VG_{\ge 0}^r$, the restriction of $f$ to $\{x\} \times A$
 is a torsion element of $\cCsq(\{x\} \times A)$ (meaning that there exists a non-zero integer $N$ such that $N\cdot f(x) = 0$).
 \item[\FT] $f$ is a torsion element of $\cCsq(\VG_{\ge 0}^r \times A)$.
 \item[\CT] For all $(a,b) \in L$, $c_{a,b}$ is a torsion element of $\cCsq(A)$.
\end{enumerate}

We shall prove the following two variants of the furthermore part of \cite[Proposition~4.1]{CH-eval}:

\begin{prop}\label{p.4.1}
For $A, f, L,\dots$ as in Assumption~\ref{a.f}, we have the following:
\begin{enumerate}
 \item If $L \subset \{0\} \times \ZZ^r$, then $\EZ \Leftrightarrow \FZ \Leftrightarrow \CZ$.
 \item $\ET \Leftrightarrow \FT \Leftrightarrow \CT$
\end{enumerate}
\end{prop}

So: The original version $\FZ\Rightarrow\CZ$ of the furthermore part holds under an additional assumption on $L$; without assumption on $L$, we at least have the torsion variant $\FT\Rightarrow\CT$.
Moreover, we already moved large parts of the proof of Proposition~\ref{p.T1} into Proposition~\ref{p.4.1} by proving the stronger conclusions $\EZ\Rightarrow\CZ$ and $\ET\Rightarrow\CT$.

The proof of Proposition~\ref{p.T1} needs one more ingredient:

\begin{lem}\label{l}
For $A, f, L,\dots$ as in Assumption~\ref{a.f}, if $\EZ$ and $\CT$ hold, then
for each $b$, we have $\sum_{a} c_{a,b}g^a = 0$. In particular, $f = 0$.
\end{lem}

\begin{proof}[Proof of Proposition~\ref{p.T1}]
First note that $\EZ$ trivially implies $\ET$; then apply Proposition~\ref{p.4.1}(2) to get $\CT$, and finally use the lemma to deduce $\FZ$.
\end{proof}

Before proceeding with the main proofs, we give an example showing that the implication $\FZ\Rightarrow\CZ$ is most likely really wrong in general (namely, as soon as $\cC(A)$ contains torsion elements).

\begin{example}\label{ex:non-torsion}
Suppose $\cC(A)$ contains a non-zero element $c$ which is $2$-torsion, that is, $2c=0$, for some $\cS$-definable set $A$, and consider the function $f \in \cC(\VG_{\geq 0}\times A)$ given by $f = cg + cg^2$ (where as usual, $g\colon \VG_{\geq 0}\times A \to \VG_{\geq 0}$ is the coordinate projection). Clearly, $f$ does not satisfy $\CZ$, but
one easily verifies that $f$ satisfies $\EZ$, and hence also $\FZ$ (e.g. by Lemma~\ref{l}).
\end{example}


In the next section, we prove Proposition~\ref{p.4.1} and Lemma~\ref{l}. As already announced, the other place where the furthermore part of \cite[Proposition~4.1]{CH-eval} is used (namely \cite[Proposition~4.2]{CH-eval}) will be treated in Section~\ref{s.int}.

\section{Proofs concerning nullity}
\label{s.proof}

We start by giving a proof of Proposition~\ref{p.4.1}(2). We will not give a separate proof of
Proposition~\ref{p.4.1}(1), since it has exactly the same structure; one just needs to replace ``is torsion'' by ``is zero'' in the obvious places and drop the part of the proof where $L$ is not a subset of $\{0\} \times \ZZ^r$. (In that part, torsion cannot be avoided.)

\begin{proof}[Proof of Proposition~\ref{p.4.1}(2)]
It suffices to prove $\ET \Rightarrow \CT$, since we trivially have $\CT \Rightarrow \FT \Rightarrow \ET$.

It is enough to prove the $r=1$ case. Indeed, the general case can be reduced to that case using induction as follows:
Set $A' = \VG_{\ge 0}^{r-1} \times A$ and write $f$ as a sum of the form \eqref{sum} ``relative to $A'$'', i.e.,
\begin{equation}\label{eq.red1}
f = \sum_{(a_1, b_1) \in L'} c'_{a_1,b_1} \cdot g^{a_1}\cdot \LL^{b_1\cdot g_1},
\end{equation}
for some suitable $L' \subset \NN \times \ZZ$ and some $c'_{a_1,b_1} \in \cCsq(A')$. More precisely, we have
\begin{equation}\label{eq.red2}
c'_{a_1,b_1} =  \sum_{(a_2, \dots, a_n, b_2,\dots, b_n)} c_{(a_1,\dots a_n),(b_1,\dots b_n)} \cdot g_2^{a_2}\cdots g_n^{a_n}\cdot \LL^{b_2\cdot g_2+\dots + b_n \cdot g_n }.
\end{equation}
By assumption, $f_{|\{x\} \times A}$ is torsion for point $x$ on $\VG_{\ge 0}^r$.
Fixing a point $x_1$ on $\VG_{\ge0}$ and applying the inductive $\ET \Rightarrow \FT$ to the restriction $f_{|\{x_1\} \times \VG_{\ge 0}^{r-1} \times A}$ yields that $f_{|\{x_1\} \times \VG_{\ge0}^{r-1} \times A}$ is torsion. Doing this for all $x_1$ shows that $f$ satisfies $\ET$ relative to $A' = \VG_{\ge 0}^{r-1} \times A$.
We now apply the $r=1$ case of $\ET \Rightarrow \CT$ to the sum \eqref{eq.red1} to get
that $c'_{a_1,b_1}$ is torsion for all $(a_1, b_1) \in L'$.
Finally, we apply the inductive $\FT \Rightarrow \CT$ to each $c'_{a_1,b_1}$, which yields that $c_{a,b}$ is torsion, as desired.

So from now on, we assume $r = 1$. We may additionally assume that $L \subset \NN \times \NN$; otherwise, multiply $f$ by a suitable power of $\LL^g$, which is invertible in $\cCsq(\VG_{\ge 0}^r \times A)$. Clearly, we may also assume that $L \ne \emptyset$.

We do an induction on $(a\maxx, b\maxx) :=  \max L$, with respect to the lexicographic order.\footnote{Note that the order of induction is different than in \cite[p.~23]{CH-eval}. This is not essential, but makes the proof cleaner.}

In the case $(a\maxx, b\maxx) = (0,0)$, the only coefficient is $c_{0,0}$, and we essentially have $c_{0,0} = f_{|\{0\}\times A}$ (where ``essentially'' means: up to identifying $A$ with $\{0\} \times A$), so we trivially have $\ET \Rightarrow \CT$.

Suppose next that $a\maxx = 0$ and $b\maxx \ge 1$.
We consider $f\neww \in \cCsq(\VG_{\ge 0} \times A)$, which, in slightly sloppy notation, is defined by $f\neww(x,w) := f(x+1,w) - \LL^{b\maxx}f(x, w)$. Writing out
the sum defining $f$ yields
\[
f\neww = \sum_{(0,b) \in L} c_{0,b}(\LL^b-\LL^{b\maxx})\LL^{bg}.
\]
Since the $(0,b\maxx)$ summand cancels, we can apply induction to $f\neww$ (note that $\ET$ for $f$ implies $\ET$ for $f\neww$) and obtain that for each $b < b\maxx$, the coefficient $c_{0,b}(\LL^b-\LL^{b\maxx})$ of $f\neww$ is torsion. Since $\LL^b-\LL^{b\maxx}$ is invertible in $\cCsq(A)$, we obtain that $c_{0,b}$ is torsion for all $b < b\maxx$. It remains to show that $c_{0,b\maxx}$ is torsion. To this end, consider
$f_{|\{0\} \times A} = \sum_b c_{0,b}$. Since $f_{|\{0\} \times A}$ is torsion by assumption, and since all but one summand of the sum are also torsion, so is the last remaining summand $c_{0,b\maxx}$.

Finally, suppose that $a\maxx > 0$. (This part is irrelevant for the proof of Proposition~\ref{p.4.1}(1).) We again consider $f\neww(x,w) := f(x+1,w) - \LL^{b\maxx}f(x, w)$, and denote
its coefficients by $c'_{a,b}$, i.e.,
$f\neww = \sum_{a,b} c'_{a,b} g^a \LL^{bg}$.
Again, $\ET$ holds for $f\neww$, and since $c'_{a\maxx b\maxx} = 0$, we can apply induction and hence obtain that all $c'_{a,b}$ are torsion. We have
\begin{align*}
c'_{a\maxx-1,b\maxx} &=
\underbrace{a\maxx\cdot c_{a\maxx,b\maxx}\LL^{b\maxx}
+ c_{a\maxx-1,b\maxx}\LL^{b\maxx}}_{\text{contribution from }f(x+1,w)} - \underbrace{c_{a\maxx-1,b\maxx}\LL^{b\maxx}}_{\!\!\!\!\!\!\text{contribution from }\LL^{b\maxx}f(x,w)\!\!\!\!\!\!}\\
&=a\maxx\cdot c_{a\maxx,b\maxx}\LL^{b\maxx},
\end{align*}
so from $c'_{a\maxx-1,b\maxx}$ being torsion, we obtain that $c_{a\maxx,b\maxx}$ is torsion. (Here, we use that $\LL^{b\maxx }$ is invertible.)
To get that the remaining coefficients of $f$ are also torsion, apply induction once more, this time to $f - c_{a\maxx,b\maxx}g^{a\maxx}\LL^{b\maxx g}$.
\end{proof}


\begin{proof}[Proof of Lemma~\ref{l}]
Set $f_b := \sum_a c_{a,b}g^a$ (so that $f = \sum_b f_b\LL^{b\cdot g}$).
Recall that we assume $\EZ$ and $\CT$ and want to prove that $f_b = 0$.

Fix a positive integer $N$ such that $N\cdot c_{a,b} = 0$ for all $(a, b) \in L$.
For each $i=(i_1, \dots, i_{r}) \in \{0,\ldots,N-1\}^r$, let $s_i$ be the $\cS$-definable map $\VG_{\ge0}^r \times A\to \VG_{\ge0}^r \times A$ sending $(x,w) \in \NN^r$ to $(Nx + i, w)$,
and let $f_i := s_i^{*}(f)$ and $f_{i,b} := s_i^{*}(f_b)$ be the pullbacks. Then $\EZ$ for $f$ clearly implies $\EZ$ for $f_i$, i.e., $f_{i|\{x\} \times A} = 0$ for every point $x$ on $\VG_{\ge 0}^r$. Note that
\begin{equation}\label{eq.f_i}
f_i = \sum_b f_{i,b}\cdot\LL^{b\cdot  (Ng + i)}
\end{equation}
and
\[
f_{i,b} = \sum_{a}  c_{a,b}\cdot  (Ng + i)^a ,
\]
(where we use some sloppy notation: $Ng + i$ is the map sending a point $(x,w)$ on $\VG^r_{\ge 0} \times A$ to the point $Nx + i$ of $\VG^r_{\ge 0}$).
After multiplying out the powers $(Ng + i)^a$ and dropping all those summands which have a factor $N$ (they disappear due to the factor $c_{a,b}$), one obtains
\[
f_{i,b} = \sum_{a}  c_{a,b}\cdot  i^a,
\]
which does not depend on $g$ anymore. (More formally, $f_{i,b}$ is the pullback
of a motivic function on $A$ along the projection $\VG^r_{\ge 0} \times A \to A$.)
Thus Proposition~\ref{p.4.1}(1) applies to \eqref{eq.f_i} and yields, using $\EZ \Rightarrow \FZ$, that $f_{i,b} = 0$.
Since this holds for all $i$, this implies $f_b = 0$, as claimed.
\end{proof}

In the next section, we will need a lemma whose proof is quite similar to the above one. We prove it right away, while the reader (hopefully) still has the above proof in mind:

\begin{lem}\label{l2}
Suppose that all structures in $\cS$ are elementarily equivalent (so that we can identify the set of points on $\VG_{\ge 0}$ with $\NN$) and
that $f = \sum_{a} c_a g^a$ is a motivic function on $\VG_{\ge 0}$
satisfying $\CT$ (i.e., each $c_a$ is torsion). Suppose moreover that
$\EZ$ holds almost everywhere, in the sense that the restriction $f_{|\{x\} \times A}$ is zero for all but finitely many $x \in \NN$. Then $f = 0$.
\end{lem}

\begin{proof}
We start as in the proof of Lemma~\ref{l}:
Fix $N > 0$ such that $N\cdot c_{a} = 0$ for all $a$, define
$s_i\colon \VG_{\ge0}\to \VG_{\ge0}, x \mapsto Nx + i$ as before for $i \in \{0, \dots, N-1\}$ and let $f_i := s_i^{*}(f)$ be the pullback.
The computation we did for $f_{i,b}$ in the above proof now applies directly to $f_i$ and yields $f_i = \sum_a c_a\cdot i^a$.
Since $\EZ$ holds for $f$ almost everywhere, it in particular holds at some points in the image of $s_i$, so there exists at least one $x \in \NN$ such that $f_{i|\{x\} \times A} = 0$. Since $f_i$ does not depend on $x$, this implies $f_i = 0$. Since this holds for all $i$, we get $f = 0$.
\end{proof}

\section{Integrability}
\label{s.int}

The second place where the furthermore part of \cite[Proposition~4.1]{CH-eval} is used (namely, apart from \cite[Theorem~1]{CH-eval}) is \cite[Proposition~4.2]{CH-eval}, which in turn is used in the proof of \cite[Theorem~2]{CH-eval}.
That Proposition~4.2 lists some statements which are claimed to be equivalent to some motivic function being integrable in relative dimension $0$. Statement (iii) from the proposition however is too strong to be equivalent.
Proposition~\ref{p.4.2} below is a corrected version of the equivalence between integrability and (iii), but for simplicity, this corrected version is stated in a less general setting than \cite[Proposition~4.2]{CH-eval}. That less general version is enough for the application in the proof of \cite[Theorem~2]{CH-eval}. The notion of integrability for motivic functions is recalled in \cite[Section 3.4]{CH-eval}, following \cite{CLoes} and \cite{CLexp}.

%
%
%
%
%
%

\begin{prop}\label{p.4.2}
Let $A, f, L,\dots$ be as in Assumption~\ref{a.f}, and suppose moreover that $A$ is of the form $A = \RF^n \times Z$, for some $\cS$-definable set $Z$.
Set $L_0 := L \cap (\NN^r \times (\ZZ\setminus\NN)^r)$.
Then the following are equivalent:
\begin{enumerate}
 \item $f$ is integrable in relative dimension $0$ over $Z$, along the projection $X \to Z$.
 \item The sub-sum $f_\infty := \sum_{(a,b) \in L \setminus L_0}  c_{a,b}\cdot  g^a \cdot  \LL^{b\cdot  g}$ is equal to $0$ (as an element of $\cCsq(\VG_{\ge 0}^r \times A)$).
\end{enumerate}
\end{prop}

The stronger statement \cite[Proposition~4.2]{CH-eval}(iii) claimed, instead of $f_\infty = 0$, that $c_{a,b} = 0$ for all $(a,b) \in L \setminus L_0$.
Our above Example~\ref{ex:non-torsion} is also a potential counter-example to that: The function $f$ in that example is integrable since it is zero, but $c_{1,0}$ and $c_{2,0}$ are non-zero.

\begin{remark}
In the more general setting of \cite[Proposition~4.2]{CH-eval}, $f$ itself is not required to be of the shape from Assumption~\ref{a.f}.
However, one can bring any motivic function $f$ into that shape by partitioning the domain into finitely many pieces $X_i$ and pulling back along certain affine linear maps $\theta_i$. The strong version of the corrected \cite[Proposition~4.2]{CH-eval} would be: For every choice of such $X_i$ and $\theta_i$,
the original $f$ satisfies Proposition~\ref{p.4.2}(1) if and only if
each pullback $\theta_i^*(f)$ satisfies
Proposition~\ref{p.4.2}(2).
That full version easily follows from the version given in Proposition~\ref{p.4.2}.\end{remark}

Before proving Proposition~\ref{p.4.2}, we explain how it can replace the wrong part of \cite[Proposition~4.2]{CH-eval} in the proof of \cite[Theorem~2]{CH-eval}.
Note that we can freely use \cite[Theorem~1]{CH-eval} and its corollaries.
The place where that part of \cite[Proposition~4.2]{CH-eval} is used
is in Case~1 on \cite[p.~25]{CH-eval}. There, after some reduction, we are in the situation where we need to prove the following:

\begin{prop}\label{p.T2}
Let $A, f, L,\dots$ be as in Assumption~\ref{a.f}. Suppose moreover that $A = \RF^n \times Z$ and that $f_{|\VG^r_{\ge 0} \times \RF^n \times\{z\}}$ is integrable for every point $z$ on $Z$. Then $f$ is integrable in relative dimension $0$ over $Z$ (along the projection $\VG^r_{\ge 0} \times \RF^n \times Z \to Z$).
\end{prop}

Here is a proof of this using Proposition~\ref{p.4.2} instead of \cite[Proposition~4.2]{CH-eval}.

\begin{proof}[Proof of Proposition~\ref{p.T2}]
Applying Proposition~\ref{p.4.2} to $f_{|\VG^r_{\ge 0} \times \RF^n \times\{z\}}$ yields that the sub-sum
\[
\sum_{(a,b) \in L \setminus L_0}  c_{a,b|\RF^n \times \{z\}}\cdot  g^a \cdot  \LL^{b\cdot  g}
\]
(for $L_0$ as in Proposition~\ref{p.4.2}) is equal to $0$.
Since this holds for every point $z$ on $Z$, \cite[Corollary~3.6.5]{CH-eval} yields that also
\[
\sum_{(a,b) \in L \setminus L_0}  c_{a,b}\cdot  g^a \cdot  \LL^{b\cdot  g}
\]
is equal to $0$. Thus we have $f = \sum_{(a,b) \in L_0} c_{a,b}g^a\LL^{bg}$, and this is clearly integrable.
\end{proof}

In the proof of Proposition~\ref{p.4.2}, we will need the following lemma:

\begin{lem}\label{l.almost0}
Let $f = \sum_{a} c_a g^a$ be a motivic function on $\VG_{\ge 0} \times A$, and denote by $h\colon \VG_{\ge 0} \times A \to A$ the projection.
Suppose that there exists an $\cS$-definable set $X \subset \VG_{\ge 0} \times A$ such that for each point $w$ on $A$, the fiber complement $h^{-1}(w)  \setminus X$ is finite and such that $f_{|X} = 0$. Then $f = 0$.
\end{lem}

\begin{proof}
By \cite[Corollary~3.6.5]{CH-eval}, it suffices to prove that $f_{|h^{-1}(w)} = 0$ for every point $w$ on $A$. Thus it suffices to prove the lemma for $A$ replaced by $w$ and $\cS$ replaced by $\cS(w)$, meaning firstly that $A$ can be entirely dropped from the notation and secondly that all stuctures in $\cS$ are elementarily equivalent, so that we can identify the set of points on $\VG_{\ge 0}$ with $\NN$.

Fix any $d \in \NN$ such that $i \in X$ for all $i \ge d$, and define $\theta\colon \VG_{\ge 0} \to \VG_{\ge 0}$ by $\theta(x) = x + d$.
Then the pullback $\theta^*(f) = \sum_a c_a (g+d)^a$ is entirely zero.
By Proposition~\ref{p.4.1} (1), $\theta^*(f)$ has torsion coefficients.
The coefficient of $g^a$ in $\theta^*(f)$ has the form $c_a + \sum_{a'>a} s_{a'}c_{a'}$ for some integers $s_{a'}$, so using a downwards induction on $a$, one obtains that all the $c_a$ are also torsion.
Now Lemma~\ref{l2} implies that $f = 0$.
\end{proof}

\begin{proof}[Proof of Proposition~\ref{p.4.2}]
First note that the implication (2) $\Rightarrow$ (1) is trivial:
By (2), $f$ can be written as $f = \sum_{(a,b) \in L_0}  c_{a,b}\cdot  g^a \cdot  \LL^{b\cdot  g}$, which is integrable essentially by definition.
So let us now prove (1) $\Rightarrow$ (2).

\medskip

We start by treating the case $r = 1$.

Essentially by definition of integrability (and using the correct part of \cite[Proposition~4.1]{CH-eval}), 
we find a partition of the domain $\VG_{\ge 0} \times A$ of $f$ into sets $X_i$ and
definable bijections $\theta_i\colon \VG^{r_i} \times A_i \to X_i$ such that
each pullback $f_i := \theta_i^*(f)$ is of the form
\begin{equation}\label{eq.sum.d}
f_i = \sum_{(a,b) \in L_i} d_{a,b,i} g^a \LL^{b\cdot g}
\end{equation}
with $L_i \subset \NN^{r_i} \times (\ZZ \setminus \NN)^{r_i}$. Moreover,
we have $r_i \in \{0,1\}$ (since we are assuming $r = 1$), and in the case $r_i=1$, the map $\theta_i$ is of the form
\[
\theta_i(x, w) = (e_ix+d_i(w), w).
\]
for some positive integer $e_i$ and some definable function $d_i\colon A_i \to \VG_{\ge 0}$.

Now recall that $f$ was given as a sum $f = \sum_{b} f_b\LL^{b\cdot g}$, where
$f_b := \sum_a c_{a,b} g^a$, and that our goal is to prove that
$f_\infty = \sum_{b \ge 0} f_b\LL^{b\cdot g}$ is zero.
We will more precisely prove that $f_b = 0$ for every $b\ge 0$.
To this end, it suffices to prove that for each $i$ with $r_i = 1$, the pullback
$\theta_i^*(f_b)$ is zero. Indeed, this implies $f_{b|X} = 0$, where
$X$ is the union the corresponding images $X_i$, and that union $X$
satisfies the assumption of Lemma~\ref{l.almost0}, which then yields $f_b = 0$.

We now fix an $i$ with $r_i = 1$ for the remainder of the proof of the $r=1$ case. The pullback of $f_b$ is of the form
\[
f_{i,b} = \theta_i^*(f_b) = \sum_a c'_{a,b,i} g^a,
\]
and expressing the pullback of $f$ in terms of this gives
\begin{equation}\label{eq.sum.c'}
f_i = \sum_b f_{i,b} \LL^{b\cdot (e_ig + d_i)}
=  \sum_b \LL^{b\cdot d_i} f_{i,b} \LL^{b e_ig}
\end{equation}

Subtracting \eqref{eq.sum.d} from \eqref{eq.sum.c'} yields a sum
\[
g = \sum_{a,b} d'_{a,b} g^a\LL^{be_i}
\]
which is zero (as a motivic function). In particular, $g$ satisfies $\EZ$, and
by Proposition~\ref{p.4.1}(2), its coefficients $d'_{a,b}$ are torsion, so Lemma~\ref{l} implies that for each fixed $b$ we have
\[
\sum_{a } d'_{a,b} g^a  = 0.
\]
Using that $d'_{a,be_i} = \LL^{bd_i}c'_{a,b,i}-d_{a,be_i,i}$, we obtain
\begin{equation}\label{eq.diff}
\sum_{a } ( \LL^{bd_i}c'_{a,b,i}-d_{a,be_i,i}) g^a  = 0
\end{equation}
for every $b$. Now recall that for $b \ge 0$, we have $d_{a,be_i,i} = 0$,
so after multiplying by $\LL^{-bd_i}$, \eqref{eq.diff} becomes $f_{i,b} = 0$, which is what we had to show in the case $r = 1$.

\medskip

Let us now come back to general $r$. To simplify notation, we set $c_{a,b} = 0$ for all $(a,b) \notin L$, so that we can omit $L$ from the notation.
We first prove a variant of the desired result (1) $\Rightarrow$ (2), namely that
(1) implies
\begin{equation}\label{eq.2var}
\sum_{(a,b) \in \NN^r \times (\NN \times \ZZ^{r-1})} c_{a,b} g^a\LL^{bg} = 0.
\end{equation}

To obtain this, set $A' = \VG_{\ge 0}^{r-1} \times A$ and write $f$ as
\begin{equation}\label{eq.int1}
f = \sum_{(a_1, b_1) \in \NN \times \ZZ} c'_{a_1,b_1} \cdot g^{a_1}\cdot \LL^{b_1\cdot g_1},
 \end{equation}
where
\[
c'_{a_1,b_1} =  \sum_{(a_2, \dots, a_n, b_2,\dots, b_n) \in \NN^{n-1} \times \ZZ^{n-1}} c_{(a_1,\dots a_n),(b_1,\dots b_n)} \cdot g_2^{a_2}\cdots g_n^{a_n}\cdot \LL^{b_2\cdot g_2+\dots + b_n \cdot g_n }.
\]
That $f$ is integrable over $Z$ implies that it is also integrable over $A'$,
so by the $r = 1$ case applied to \eqref{eq.int1}, we have
\[
\sum_{(a_1,b_1) \in \NN \times \NN} c'_{a_1,b_1}g^{a_1}\LL^{b_1g_1} = 0.
\]
Plugging in the $c'_{a_1,b_1}$ yields \eqref{eq.2var}.

\medskip

Finally, we are in the position to prove (1) $\Rightarrow$ (2) for arbitrary $r$. To this end,
set
\[
f_i := \sum_{(a,b) \in M_i} c_{a,b}g^a\LL^{bg},
\]
where $M_i = \NN^r \times (\ZZ \setminus \NN)^i \times \ZZ^{r-i}$, for $i = 0, \dots, r$. Then $f_r = f - f_\infty$, so to obtain $f_\infty = 0$, we can equivalently show that $f_r = f$.

Clearly we have $f_0 = f$. For $i \ge 1$, we have
\[
f_{i-1} - f_i =
\sum_{(a,b) \in M_{i-1} \setminus M_i} c_{a,b}g^a\LL^{bg},
\]
and $M_{i-1} \setminus M_i = \{(a,b) \in M_{i-1} \mid b_i \in \NN\}$, so
``(1) $\Rightarrow$ \eqref{eq.2var}'' applied to $f_{i-1}$ (with the first and the $i$th coordinate swapped) yields that $f_{i-1} - f_i = 0$.
Putting everything together gives $f = f_0 = \dots = f_r$.
\end{proof}

The authors would like to apologise for any inconvenience caused.


\bibliographystyle{amsplain}
\bibliography{anbib}

\end{document}